\documentclass{article}
\usepackage{amssymb}
\usepackage{amsmath}

\newtheorem{theorem}{Theorem}
\newtheorem{corollary}{Corollary}
\newtheorem{lemma}{Lemma}
\newtheorem{remark}{Remark}
\newenvironment{proof}[1][Proof]{\noindent\textbf{#1.} }{\ \rule{0.5em}{0.5em}}

\begin{document}

\title{\textbf{Pointwise }$\left( H,\Phi \right) $ \textbf{strong\
approximation by Fourier series of }$L^{\Psi }$ \textbf{integrable functions}%
}
\author{\textbf{W\l odzimierz \L enski } \\
University of Zielona G\'{o}ra\\
Faculty of Mathematics, Computer Science and Econometrics\\
65-516 Zielona G\'{o}ra, ul. Szafrana 4a, Poland\\
W.Lenski@wmie.uz.zgora.pl }
\date{}
\maketitle

\begin{abstract}
We essentially extend and improve the classical result of G. H. Hardy and J.
E. Littlewood on strong summability of Fourier series. We will present an
estimation of the generalized strong mean $\left( H,\Phi \right) $ as an
approximation version \ of \ the Totik type generalization \ of the result
of \ G. H. Hardy, J. E. Littlewood, in case of integrable functions from $%
L^{\Psi }$. As a measure of such approximation we will use the function
constructed by function $\Psi $ complementary to $\Phi $ on the base of
definition of the $L^{\Psi }$ points. Some corollary and remarks will also
be given.

\ \ \ \ \ \ \ \ \ \ \ \ \ \ \ \ \ \ \ \ 
\end{abstract}

\footnotetext[1]{%
Key words: Strong approximation, rate of pointwise strong summability,
Orlicz spaces}

\footnotetext[2]{%
2000 Mathematics Subject Classification: 42A24}

\newpage

\section{Introduction}

Let $L^{p}\ (1\leq p<\infty )\;$be the class of all $2\pi $--periodic
real--valued functions integrable in the Lebesgue sense with $p$--th power
over $Q=$ $[-\pi ,\pi ]$.

A mapping $\Phi :%
\mathbb{R}
\rightarrow 
\mathbb{R}
^{+}$ is termed an $N-function$ if $\left( i\right) $ $\Phi $ is even and
convex, $\left( ii\right) $ $\Phi \left( u\right) =0$ iff $u=0$, $\left(
iii\right) $ $\lim\limits_{u\rightarrow 0}\frac{\Phi \left( u\right) }{u}=0,$
\ $\lim\limits_{u\rightarrow \infty }\frac{\Phi \left( u\right) }{u}=\infty
. $ The left derivative $p(t)=\Phi ^{\prime }(t)$\ exists and is left
continuous, nondecreasing on $(0,\infty )$, satisfies\ $0<p(t)<\infty $ for $%
0<t<\infty $, $p(0)=0$ and $\lim_{t\rightarrow \infty }p(t)=\infty $.\ The
left inverse $q$ of $p$ is, by definition, $q(s)=\inf
\{t>0:p(t)>s\}=p^{-1}\left( s\right) $\ for $s>0$.\ Then $\Phi $ and $\Psi $
given by $\Phi (u)=\int_{0}^{\left\vert u\right\vert }p\left( t\right) dt$
and $\Psi (v)=\int_{0}^{\left\vert v\right\vert }q\left( s\right) ds$ are
called a pair of complementary $N-functions$ which satisfy the Young
inequality: $\left\vert uv\right\vert \leq \Phi (u)+\Psi (v).$\ The $%
N-function$ $\Psi $ complementary to $\Phi $ can equally be defined by: $%
\Psi (v)=\sup \{u|v|-\Phi (u):u\geq 0\},$ $v\in 
\mathbb{R}
.$ An example of complementary pair of $N-functions$ is following one: $\Phi
(u)=e^{\left\vert u\right\vert }-\left\vert u\right\vert -1,$ $\Psi
(v)=(1+\left\vert v\right\vert )\log (1+\left\vert v\right\vert )-\left\vert
v\right\vert .$ Using function $\Psi $ we can define the Orlicz space $%
L^{\Psi }=\left\{ f:\int_{_{_{Q}}}\Psi \left( \left\vert f(x)\right\vert
\right) dx<\infty \right\} $.

Consider the trigonometric Fourier series\ 
\begin{equation*}
Sf(x)=\frac{a_{o}(f)}{2}+\sum_{k=1}^{\infty }(a_{k}(f)\cos kx+b_{k}(f)\sin
kx)=\sum_{k=0}^{\infty }c_{k}(f)\exp ikx
\end{equation*}%
and denote by $S_{\nu }f$ \ the partial sums of $Sf$. Then, 
\begin{equation*}
H_{n}^{\Phi }f\left( x\right) :=\Phi ^{-1}\left\{ \frac{1}{n+1}\sum_{\nu
=0}^{n}\Phi \left( \left\vert S_{\nu }f\left( x\right) -f\left( x\right)
\right\vert \right) \right\}
\end{equation*}%
($H_{n}^{\Phi }f=H_{n}^{q}f$\ when $\Phi \left( u\right) =u^{q}$).

As a measure of the above deviation we will use the pointwise characteristic 
$(G_{p,\Psi }-points)$%
\begin{equation*}
G_{x}f\left( \delta \right) _{p,\Psi }:=\Psi ^{_{^{-1}}}\left\{ \sum_{k=1}^{ 
\left[ \pi /\delta \right] }\Psi \left[ \left( \frac{1}{k\delta }%
\int_{\left( k-1\right) \delta }^{k\delta }\left\vert \varphi _{x}\left(
t\right) \right\vert ^{p}dt\right) ^{1/p}\right] \right\} ,\text{ \ \ }p\geq
1,
\end{equation*}%
where \ \ $\varphi _{x}\left( t\right) :=f\left( x+t\right) +f\left(
x-t\right) -2f\left( x\right) ,$ constructed by function $\Psi $
complementary to $\Phi $ on the base of the following definition of the
Gabisonia points $(G_{p,s}-points)$ \cite{1}%
\begin{equation*}
\lim_{n\rightarrow \infty }G_{x}f\left( \frac{\pi }{n+1}\right) _{p,s}=0,
\end{equation*}%
where%
\begin{equation*}
G_{x}f\left( \delta \right) _{p,s}:=\left\{ \sum_{k=1}^{\left[ \pi /\delta %
\right] }\left( \frac{1}{k\delta }\int_{\left( k-1\right) \delta }^{k\delta
}\left\vert \varphi _{x}\left( t\right) \right\vert ^{p}dt\right)
^{s/p}\right\} ^{1/s},\text{ \ \ }s>p\geq 1,
\end{equation*}%
and the characteristic $(L^{\Psi }-points)$%
\begin{equation*}
w_{x}f\left( \delta \right) _{\Psi }:=\Psi ^{_{^{-1}}}\left\{ \frac{1}{%
\delta }\int_{0}^{\delta }\Psi \left( \left\vert \varphi _{x}\left( t\right)
\right\vert \right) dt\right\}
\end{equation*}%
constructed on the base of definition of the Lebesgue points $(L^{p}-points)$
defined as usually by the formula%
\begin{equation*}
\lim_{n\rightarrow \infty }w_{x}f(\frac{\pi }{n+1})_{p}=0,
\end{equation*}%
where%
\begin{equation*}
w_{x}f(\delta )_{p}:=\left\{ \frac{1}{\delta }\int_{0}^{\delta }\left\vert
\varphi _{x}\left( t\right) \right\vert ^{p}dt\right\} ^{1/p}.
\end{equation*}

It is well-known that $H_{n}^{q}f\left( x\right) -$ means tend to $0$ $%
\left( \text{as}\ n\rightarrow \infty \right) $ at the $L^{p}-points$ $x$ of 
$f$ $\in L^{p}$ $\left( 1<p<\infty \right) $ and at the $G_{1,s}-points$ $x$
of $f\in L^{1}$ $\left( s>1\right) .$ These facts were proved as a
generalization of the Fej\'{e}r classical result on the convergence of the $%
\left( C,1\right) $ -means of Fourier series by G. H. Hardy and J. E.
Littlewood in \cite{2, 3} and by O. D. Gabisoniya in \cite{1}, respectively.
In case \ $L^{1}$ and convergence almost everywhere the first results on
this area were due to J. Marcinkiewicz \cite{JM}\ and A. Zygmund \cite{ZA}.
It is also clear, as it shown L. Gogoladze \cite{LG} and W. A. Rodin \cite{R}%
, that $H_{n}^{\Phi }f\left( x\right) -$means, with $\Phi \left( u\right)
=\exp u-1$ also tend to $0$ almost everywhere $\left( \text{as }n\rightarrow
\infty \right) $. The estimates of $H_{n}^{\Phi }f\left( x\right) -$means,
for $f$ $\in L^{p}$ $\left( 1<p<\infty \right) $ was obtained in \cite{WL}
but in the case $f\in L^{1}$ the estimates of $H_{n}^{\Phi }f\left( x\right) 
$ $-$ means with $\Phi \left( u\right) =u^{q}$ was obtained in \cite{10, NR}%
. Finally, the estimation of the $H_{n}^{\Phi }f\left( x\right) -$ mean for $%
f\in L^{1}$ by $G_{x}f\left( \frac{\pi }{n+1}\right) _{1,\Psi }$\ as
approximation version of the Totik type $\left( \text{see \cite{111, 12}}%
\right) $ generalization of the mentioned results of \ J. Marcinkiewicz\ and
A. Zygmund was obtained in \cite{100} as follows:

\textbf{Theorem }\textit{Let }$\Phi ,\Psi \in \digamma $\textit{\ are
complementary pair of }$N-functions$\textit{, such that }$\Psi (x)/x$\textit{%
\ be non-decreasing and }$\Psi (x)/x^{2}$\textit{\ non-increasing, and let }$%
p$\textit{\ be convex, }$q$\textit{\ be non-negative, continuous, and
strictly increasing such that }$\frac{q\left( s\right) }{s}$\textit{\ is
non-increasing. If }$f\in L^{1}$\textit{, then }%
\begin{equation*}
H_{n}^{\Phi }f\left( x\right) \ll G_{x}f\left( \frac{\pi }{n+1}\right)
_{1,\Psi }\text{,}
\end{equation*}%
\textit{for\ }$n=0,1,2,.,$\textit{\ where }%
\begin{equation*}
\digamma =\left\{ \Phi :\Phi \text{ is }N-function\text{ with convex or
concave left derivative }p\right\} .
\end{equation*}

In this paper we will consider the function $f\in L^{\Psi }$ and the
quantity $H_{n}^{\Phi }f\left( x\right) $, but as a measure of approximation
of this quantity we will use the function constructed by function $\Psi $
complementary to $\Phi $ on the base of definition of the $L^{\Psi }$
points.\ We note that O. D. Gabisoniya in \cite{11} shows that for $f$ $\in
L^{p}$ $\left( 1<p<\infty \right) $ the relation%
\begin{equation*}
G_{x}f\left( \delta \right) _{1,p}=o_{x}\left( 1\right)
\end{equation*}%
holds at every $L^{p}-points$ $x$ of $f$. Here we will show that for $f\in
L^{\Psi }$ the relation 
\begin{equation*}
G_{x}f\left( \frac{\pi }{n+1}\right) _{1,\Psi }=o_{x}\left( 1\right)
\end{equation*}%
and thus the relation%
\begin{equation*}
H_{n}^{\Phi }f\left( x\right) =o_{x}\left( 1\right)
\end{equation*}%
hold at every $L^{\Psi }-points$ $x$ of $f.$ More precisely, we will prove
the estimate of the quantity $H_{n}^{\Phi }f\left( x\right) $ by the
characteristic constructed with $L^{\Psi }$ pointwise modulus of continuity.
Such estimate is a significant improvement and extension of the results of
G. H. Hardy and J. E. Littlewood from \cite{2, 3}. Considered here function $%
\Phi $ can be an exponential function but the space $L^{\Psi }$ can be in
between $L^{1}$ and $L^{p}$ with $p>1.$ We also give some corollary with
some example of such functions $\Phi $ and $\Psi .$ This a very sharpened
form of the conjecture of G. H. Hardy and J. E. Littlewood from \cite{4}
proved by Wang, Fu Traing in \cite{WFT}. Additionally, a remark on the
mentioned results of G. H. Hardy and J. E. Littlewood from \cite{2, 3} will
be formulated. Finally, we formulate a remark on the conjugate Fourier
series.

We shall write $I_{1}\ll I_{2}$ if there exists a positive constant $K$,
sometimes depended on some parameters, such that $I_{1}\leq KI_{2}$.\medskip

\section{Statement of the results}

Our main theorem has the following form:

\begin{theorem}
Let $\Phi ,\Psi \in \digamma $ are complementary pair of $N-functions$, such
that $\Psi \left( u\right) $ and $u^{2}$ are equivalent for small $u\geq 0,$ 
$\Psi (x)/x$ be non-decreasing and $\Psi (x)/x^{2}$ non-increasing, and let $%
p$ be convex, $q$ be continuous, and strictly increasing such that $\frac{%
q\left( s\right) }{s}$ is non-increasing and series $\sum_{k=0}^{\infty }%
\frac{1}{\left( k+1\right) ^{1/2}}q\left( \frac{1}{k+1}\right) $ is
convergence. If $f\in L^{\Psi }$, \textit{then} \ 
\begin{eqnarray*}
H_{n}^{\Phi }f\left( x\right) &\ll &\Psi ^{-1}\left[ \sum_{k=0}^{n}\frac{%
\Psi \prime \left( \frac{1}{k+1}\right) }{\left( k+1\right) ^{1/2}}\Psi
\left( w_{x}f\left( \frac{\pi \left( k+1\right) }{n+1}\right) _{\Psi
}\right) \right. \\
&&+\left. \left( n+1\right) \Psi \left( \frac{1}{n+1}\right) \Psi \left(
w_{x}f\left( \pi \right) _{\Psi }\right) \right] \text{,}
\end{eqnarray*}%
for\ $n=0,1,2,...$
\end{theorem}

From this result we can derive the following corollary.

\begin{corollary}
Let $\Phi (t)=e^{\left\vert t\right\vert }-\left\vert t\right\vert -1$ and $%
\Psi (t)=(1+\left\vert t\right\vert )\log (1+\left\vert t\right\vert
)-\left\vert t\right\vert $. If $f\in L^{\Psi }$, \textit{then } 
\begin{eqnarray*}
&&\sum_{k=0}^{n}\frac{\Psi \prime \left( \frac{1}{k+1}\right) }{\left(
k+1\right) ^{1/2}}\Psi \left( w_{x}f\left( \frac{\pi \left( k+1\right) }{n+1}%
\right) _{\Psi }\right) +\left( n+1\right) \Psi \left( \frac{1}{n+1}\right)
\Psi \left( w_{x}f\left( \pi \right) _{\Psi }\right) \\
&=&o_{x}\left( 1\right) \text{ a.e. (at }L^{\Psi }-points\text{ }x\text{)}
\end{eqnarray*}%
and thus also 
\begin{equation*}
H_{n}^{\Phi }f\left( x\right) =o_{x}\left( 1\right) \text{ a.e. (at }L^{\Psi
}-points\text{ }x\text{).}
\end{equation*}
\end{corollary}

Finally we have also two remarks.

\begin{remark}
Let $\Phi (t)=t^{\alpha }$ and $\Psi (t)=t^{\alpha /\left( \alpha -1\right)
} $ $\left( \alpha \geq 2\right) .$ For such functions the assumptions of
Theorem 1 are fulfilled. If $f\in L^{\Psi }$ \textit{then,} relations of the
before corollary hold evidently. Thus we have the mentioned results of G. H.
Hardy and J. E. Littlewood.
\end{remark}

\begin{remark}
We can observe that in the light of the O. D. Gabisoniya \cite{11} and I.
Ya. Novikov, W. A. Rodin \cite{NR} results our pointwise results remain true
for the conjugate Fourier series too.
\end{remark}

\section{Auxiliary result}

Here we present the following lemma:

\begin{lemma}
If a function $\Psi $ satisfies the conditions $\Psi \left( u\right) \gg
u^{2}$ for small $u\geq 0$ and $\Psi \left( u\right) \ll u^{2}$ for all $%
u\geq 0$, then%
\begin{equation*}
\Psi \left( \frac{u}{n+1}\right) \ll \Psi \left( \frac{1}{n+1}\right) \Psi
\left( u\right)
\end{equation*}%
for small $u\geq 0$ and $n=0,1,2,...$
\end{lemma}

\begin{proof}
Our inequality follows at once from the following inequalities 
\begin{equation*}
\Psi \left( \frac{u}{n+1}\right) \ll \left( \frac{u}{n+1}\right) ^{2}\ll
\left( \frac{1}{n+1}\right) ^{2}\Psi \left( u\right) \text{, \ for }u\ll 1.
\end{equation*}
\end{proof}

\section{Proofs of the results}

\subsection{Proof of Theorem 1}

In view of Theorem we have to prove that%
\begin{eqnarray*}
\Psi \left[ G_{x}f\left( \frac{\pi }{n+1}\right) _{1,\Psi }\right]  &\ll
&\sum_{k=0}^{n-1}\frac{1}{\left( k+1\right) ^{2}}\Psi \left( w_{x}f\left( 
\frac{\pi \left( k+1\right) }{n+1}\right) _{\Psi }\right)  \\
&&+\left( n+1\right) \Psi \left( \frac{1}{n+1}\right) \Psi \left(
w_{x}f\left( \pi \right) _{\Psi }\right) .
\end{eqnarray*}%
Let $\Delta _{\nu }^{n}=\left( \frac{\pi \nu }{n+1},\frac{\pi \left( \nu
+1\right) }{n+1}\right) $ for $\nu =0,1,2,...,n$ and%
\begin{equation*}
\sup_{\substack{ n.k \\ 0\leq k\leq n}}\left[ \frac{n+1}{\pi \left(
k+1\right) }\int_{0}^{\frac{\pi \left( k+1\right) }{n+1}}\Psi \left(
\left\vert \varphi _{x}\left( t\right) \right\vert \right) dt\right]
=M(x)<\infty \text{ (at }L^{\Psi }-points\text{ }x\text{).}
\end{equation*}%
Then, for any $\epsilon >0$ there exists a natural number $n_{\epsilon
}=n_{\epsilon }\left( x\right) $ such that $M(x)<\epsilon \sqrt{n_{\epsilon }%
},$ whence by the assumptions on $\Psi $ and $\Psi ^{\prime },$ convexity of 
$\Psi $ and the Abel transformation 
\begin{eqnarray*}
&&\sum_{k=0}^{n}\Psi \left( \frac{1}{k+1}\right) \Psi \left( \frac{n+1}{\pi }%
\int_{\Delta _{k}^{n}}\left\vert \varphi _{x}\left( t\right) \right\vert
dt\right)  \\
&\leq &\sum_{k=0}^{n}\Psi \left( \frac{1}{k+1}\right) \left( \frac{n+1}{\pi }%
\int_{\Delta _{k}^{n}}\Psi \left( \left\vert \varphi _{x}\left( t\right)
\right\vert \right) dt\right) 
\end{eqnarray*}%
\begin{eqnarray*}
&=&\frac{n+1}{\pi }\sum_{k=0}^{n-1}\left[ \Psi \left( \frac{1}{k+1}\right)
-\Psi \left( \frac{1}{k+2}\right) \right] \sum_{\nu =0}^{k}\int_{\Delta
_{\nu }^{n}}\Psi \left( \left\vert \varphi _{x}\left( t\right) \right\vert
\right) dt \\
&&+\frac{n+1}{\pi }\Psi \left( \frac{1}{n+1}\right) \sum_{\nu
=0}^{n}\int_{\Delta _{\nu }^{n}}\Psi \left( \left\vert \varphi _{x}\left(
t\right) \right\vert \right) dt
\end{eqnarray*}%
\begin{eqnarray*}
&\leq &\frac{n+1}{\pi }\sum_{k=0}^{n-1}\Psi \prime \left( \frac{1}{k+1}%
\right) \left( \frac{1}{k+1}\right) ^{2}\int_{0}^{\frac{\pi \left(
k+1\right) }{n+1}}\Psi \left( \left\vert \varphi _{x}\left( t\right)
\right\vert \right) dt \\
&&+\left( n+1\right) \Psi \left( \frac{1}{n+1}\right) \frac{1}{\pi }%
\int_{0}^{\pi }\Psi \left( \left\vert \varphi _{x}\left( t\right)
\right\vert \right) dt
\end{eqnarray*}%
\begin{eqnarray*}
&=&\sum_{k=1}^{n-1}\frac{1}{k+1}\Psi \prime \left( \frac{1}{k+1}\right) %
\left[ \frac{n+1}{\pi \left( k+1\right) }\int_{0}^{\frac{\pi \left(
k+1\right) }{n+1}}\Psi \left( \left\vert \varphi _{x}\left( t\right)
\right\vert \right) dt\right]  \\
&&+\left( n+1\right) \Psi \left( \frac{1}{n+1}\right) \frac{1}{\pi }%
\int_{0}^{\pi }\Psi \left( \left\vert \varphi _{x}\left( t\right)
\right\vert \right) dt,
\end{eqnarray*}%
and by the consideration similar to that of O. D. Gabisoniya in \cite[p. 925]%
{11}%
\begin{eqnarray*}
&&\sum_{k=0}^{n}\frac{1}{\left( k+1\right) ^{1/2}}\Psi \prime \left( \frac{1%
}{k+1}\right) \left[ \frac{n+1}{\pi \left( k+1\right) ^{3/2}}\int_{0}^{\frac{%
\pi \left( k+1\right) }{n+1}}\Psi \left( \left\vert \varphi _{x}\left(
t\right) \right\vert \right) dt\right]  \\
&\ll &\max_{0\leq k\leq n}\left[ \frac{n+1}{\pi \left( k+1\right) ^{3/2}}%
\int_{0}^{\frac{\pi \left( k+1\right) }{n+1}}\Psi \left( \left\vert \varphi
_{x}\left( t\right) \right\vert \right) dt\right] 
\end{eqnarray*}%
\begin{eqnarray*}
&\ll &\left( \max_{0\leq k\leq n_{\epsilon }}+\max_{n_{\epsilon }\leq k\leq
n}\right) \left[ \frac{n+1}{\pi \left( k+1\right) ^{3/2}}\int_{0}^{\frac{\pi
\left( k+1\right) }{n+1}}\Psi \left( \left\vert \varphi _{x}\left( t\right)
\right\vert \right) dt\right]  \\
&\ll &\left[ \frac{n+1}{\pi }\int_{0}^{\frac{\pi \left( n_{\epsilon
}+1\right) }{n+1}}\Psi \left( \left\vert \varphi _{x}\left( t\right)
\right\vert \right) dt\right] +\frac{M\left( x\right) }{\sqrt{n_{\epsilon }}}
\\
&=&o_{x}\left( 1\right) +\epsilon \text{ (at }L^{\Psi }-points\text{ }x\text{%
).}
\end{eqnarray*}%
Since $\lim\limits_{u\rightarrow 0}\frac{\Psi \left( u\right) }{u}=0,$ we
have 
\begin{equation*}
\left( n+1\right) \Psi \left( \frac{1}{n+1}\right) \frac{1}{\pi }%
\int_{0}^{\pi }\Psi \left( \left\vert \varphi _{x}\left( t\right)
\right\vert \right) dt=o_{x}\left( 1\right) ,
\end{equation*}%
and therefore 
\begin{equation*}
\sum_{k=0}^{n}\Psi \left( \frac{1}{k+1}\right) \Psi \left( \frac{n+1}{\pi }%
\int_{\Delta _{k}^{n}}\left\vert \varphi _{x}\left( t\right) \right\vert
dt\right) =o_{x}\left( 1\right) .
\end{equation*}%
This relation with the evident estimate%
\begin{equation*}
\sum_{k=0}^{\infty }\Psi \left( \frac{1}{k+1}\right) \geq \Psi \left(
1\right) >0
\end{equation*}%
yields%
\begin{equation*}
\Psi \left( \frac{n+1}{\pi }\int_{\Delta _{k}^{n}}\left\vert \varphi
_{x}\left( t\right) \right\vert dt\right) =o_{x}\left( 1\right) ,\text{ for }%
k=1,2,3,...,n,
\end{equation*}%
and 
\begin{equation*}
\Psi \left( \frac{n+1}{\pi }\int_{\Delta _{k}^{n}}\left\vert \varphi
_{x}\left( t\right) \right\vert dt\right) \ll 1,\text{ for }k=1,2,3,...,n,
\end{equation*}%
as \ well 
\begin{equation*}
\frac{n+1}{\pi }\int_{\Delta _{k}^{n}}\left\vert \varphi _{x}\left( t\right)
\right\vert dt\ll 1,\text{ for }k=1,2,3,...,n.
\end{equation*}%
Contrary, if we assume 
\begin{equation*}
\Psi \left( \frac{n+1}{\pi }\int_{\Delta _{k}^{n}}\left\vert \varphi
_{x}\left( t\right) \right\vert dt\right) \gg 1,\text{ for }k=1,2,3,...,n,
\end{equation*}%
then%
\begin{equation*}
\sum_{k=0}^{n}\Psi \left( \frac{1}{k+1}\right) \Psi \left( \frac{n+1}{\pi }%
\int_{\Delta _{k}^{n}}\left\vert \varphi _{x}\left( t\right) \right\vert
dt\right) \gg \sum_{k=0}^{n}\Psi \left( \frac{1}{k+1}\right) >0
\end{equation*}%
but it is impossible, whence the conjecture $\frac{n+1}{\pi }\int_{\Delta
_{k}^{n}}\left\vert \varphi _{x}\left( t\right) \right\vert dt\ll 1$ is true.

Hence, by the assumption, Lemma 1 gives%
\begin{equation*}
\Psi \left( \frac{n+1}{\pi \left( k+1\right) }\int_{\Delta
_{k}^{n}}\left\vert \varphi _{x}\left( t\right) \right\vert dt\right) \ll
\Psi \left( \frac{1}{k+1}\right) \Psi \left( \frac{n+1}{\pi }\int_{\Delta
_{k}^{n}}\left\vert \varphi _{x}\left( t\right) \right\vert dt\right)
\end{equation*}%
and consequently\ 
\begin{eqnarray*}
\Psi \left[ G_{x}f\left( \frac{\pi }{n+1}\right) _{1,\Psi }\right]
&=&\sum_{k=0}^{n}\Psi \left( \frac{n+1}{\pi \left( k+1\right) }\int_{\Delta
_{k}^{n}}\left\vert \varphi _{x}\left( t\right) \right\vert dt\right) \\
&\ll &\sum_{k=0}^{n}\Psi \left( \frac{1}{k+1}\right) \Psi \left( \frac{n+1}{%
\pi }\int_{\Delta _{k}^{n}}\left\vert \varphi _{x}\left( t\right)
\right\vert dt\right) .
\end{eqnarray*}%
Applying the above calculation we obtain%
\begin{eqnarray*}
&&\Psi \left[ G_{x}f\left( \frac{\pi }{n+1}\right) _{1,\Psi }\right] \\
&\ll &\sum_{k=0}^{n-1}\frac{1}{\left( k+1\right) ^{1/2}}\Psi \prime \left( 
\frac{1}{k+1}\right) \left[ \frac{n+1}{\pi \left( k+1\right) }\int_{0}^{%
\frac{\pi \left( k+1\right) }{n+1}}\Psi \left( \left\vert \varphi _{x}\left(
t\right) \right\vert \right) dt\right] \\
&&+\left( n+1\right) \Psi \left( \frac{1}{n+1}\right) \frac{1}{\pi }%
\int_{0}^{\pi }\Psi \left( \left\vert \varphi _{x}\left( t\right)
\right\vert \right) dt.
\end{eqnarray*}%
Finally, by the definition of the $w_{x}f\left( \cdot \right) _{\Psi }$ the
desired at the begin estimate follows. $\blacksquare $

\subsection{Proof of Corollary 1}

At the begin, we note that if $\Phi (t)=e^{\left\vert t\right\vert
}-\left\vert t\right\vert -1,$ $\Psi (t)=(1+\left\vert t\right\vert )\log
(1+\left\vert t\right\vert )-\left\vert t\right\vert ,$ then$\ \frac{\Psi (t)%
}{t}$\ and $q\left( t\right) =\Psi ^{\prime }(t)=\log (1+\left\vert
t\right\vert )$\ increase, series $\sum_{k=0}^{\infty }\frac{1}{\left(
k+1\right) ^{1/2}}q\left( \frac{1}{k+1}\right) $ is convergence,$\ \frac{%
\Psi (t)}{t^{2}}$ and $\frac{q\left( t\right) }{t}$\ decrease\ and $\Psi
\left( t\right) \leq t^{2}$ for all $t\geq 0$ as well $\Psi \left( t\right)
\gg t^{2}$ for small $t\geq 0.$ Therefore, by Theorem 1 and its proof, the
results follow immediately. $\blacksquare $

\end{document}